\documentclass[12pt,leqno]{amsart}

\usepackage{graphicx}
\usepackage{pinlabel} 
\usepackage{amsmath}
\usepackage{amsfonts}
\usepackage{amssymb}
\usepackage{mathtools}
\usepackage{amscd}
\usepackage[all,cmtip]{xy}
\usepackage{amsthm}
\usepackage{comment}
\usepackage{epsfig}
\usepackage[utf8]{inputenc}
\usepackage[colorinlistoftodos]{todonotes}
\usepackage[capitalise]{cleveref}
\usepackage{geometry}
\makeatletter
\providecommand\@dotsep{5}
\def\listtodoname{List of Todos}
\def\listoftodos{\@starttoc{tdo}\listtodoname}
\makeatother

\newtheorem{theorem}{Theorem}[section]
\newtheorem{proposition}[theorem]{Proposition}
\newtheorem{corollary}[theorem]{Corollary}
\newtheorem{lemma}[theorem]{Lemma}

\newtheorem{conjecture}[theorem]{Conjecture}
  \theoremstyle{definition}

\newtheorem{example}[theorem]{Example}
\newtheorem{remark}[theorem]{Remark}

\newcommand{\dbE}{\mathbb{E}}

\newcommand{\dbH}{{\mathbb H}}
\newcommand{\dbK}{\mathbb{K}}

\newcommand{\dbR}{\mathbb{R}}
\newcommand{\dbS}{\mathbb{S}}

\newcommand{\dbZ}{\mathbb{Z}}

\newcommand{\calA}{{\mathcal A}}

\newcommand{\calF}{{\mathcal F}}

\newcommand{\calG}{{\mathcal G}}
\newcommand{\calH}{{\mathcal H}}
\newcommand{\calI}{{\mathcal I}}
\newcommand{\calK}{{\mathcal K}}

\newcommand{\calM}{{\mathcal M}}

\newcommand{\calV}{{\mathcal V}}
\newcommand{\calW}{{\mathcal W}}

\newcommand{\ktheory}{$K$-theory }
\newcommand{\KR}{\mathbb{K}_R}

\newcommand{\vcyc}{V\text{\tiny{\textit{CYC}}}}

\newcommand{\fin}{F\text{\tiny{\textit{IN}}}}
\newcommand{\all}{A\text{\tiny{\textit{LL}}}}

\newcommand{\func}[3]{#1:#2\to#3}

\newcommand{\beq}{\begin{equation}}
\newcommand{\eeq}{\end{equation}}

\newcommand{\barG}{\overline{G}}

\newcommand{\evc}{\underline{\underline{E}}}
\newcommand{\efin}{\underline{E}}

\newcommand{\kleingp}{\dbZ\rtimes \dbZ}

\theoremstyle{theorem}
\newtheorem*{relfinite}{Theorem}
\newtheorem*{relvc}{Theorem}

\begin{document}

\title[K-theory of 3-manifold groups]{On the algebraic K-theory of 3-manifold groups}

\author{Daniel Juan-Pineda}
\address{Centro de Ciencias Matem\'aticas\\
Universidad Nacional Aut\'onoma de M\'exico,
Campus Morelia\\
Morelia, Michoac\'an, Mexico 58089}
\email{daniel@matmor.unam.mx}


\subjclass[2020]{Primary 19A31, 19B28, 14C35,57K32,57K31, 05C25,16S34}

\author{Luis Jorge S\'anchez Salda\~na}
\address{Departamento de Matem\' aticas, Facultad de Ciencias, Universidad Nacional Aut\'onoma de M\'exico, Circuito Exterior S/N, Cd. Universitaria, Colonia Copilco el Bajo, Delegaci\'on Coyoac\'an, 04510, M\'exico D.F., Mexico}
\email{luisjorge@ciencias.unam.mx}

\date{}


\keywords{$K$-Theory, Farrell-Jones conjecture, 3-manifold groups, Whitehead groups}

\begin{abstract}
We provide descriptions of the Whitehead  groups, and the algebraic $K$-theory groups, of the fundamental group of a connected, oriented, closed $3$-manifold in terms of Whitehead groups of their finite subgroups and certain Nil-groups. The main tools we use are: the K-theoretic Farrell-Jones isomorphism conjecture, the construction of  models for the universal space for the family of virtually cyclic subgroups in 3-manifold groups, and both the prime and JSJ-decompositions together with the well-known geometrization theorem.
\end{abstract}
\maketitle

\section{Introduction}

Let $M$ be a connected, closed, oriented $3$-manifold, and let $\Gamma=\pi_1(M)$ be the fundamental group of $M$. We say that $\Gamma$ is a \emph{$3$-manifold group}.


Let $R$ be a ring, for many years, there has been a formidable development in the study of the algebraic $K$-theory of $R$.
The case of group rings $R[\Gamma]$ has strong connections with geometry and topology and their understanding requires deep knowledge
of the group and its geometric properties. F. T. Farrell and L. Jones established in \cite{FJ93} their fundamental conjecture to try
to understand these $K$ groups for group rings. Roughly, they proposed that the groups $K_*(R[\Gamma])$ should be determined 
by the universal space for actions with virtually cyclic isotropy, $\evc \Gamma$, and homological information, see section \ref{prel} for details.

Calculations of Whitehead groups for 3-manifold groups have been a matter of study for a long time. For example, F. Waldhausen proved they vanish for certain family of 3-manifold groups in \cite{Wa78}. S. K. Roushon \cite[Theorem~1.2]{Ro11} and F. T. Farrell and L. Jones \cite{FJ85} showed that $Wh_n(\Gamma)\otimes  \dbZ[1/N]=0$, for all non-negative numbers, where $N=[(n+1)/2]!$ when $\Gamma$ is a torsion free  3-manifold group. In particular, $\widetilde K_0(\dbZ\Gamma)=Wh(\Gamma)=Wh_2(\Gamma)=0$. Moreover, $K_n(\dbZ[\Gamma])=0$ for all negative integers $n$. 

In \cite{JLSS} Joecken, Lafont and Sánchez Saldaña, computed the virtually cyclic dimension of a $3$-manifold group $\Gamma$. To achieve this, they constructed explicit models for $\underline{\underline{E}}\Gamma$. On the other hand, after the confirmation of Thurston geometrization conjecture, it is now well established that the Farrell-Jones isomorphism conjecture is valid for 3-manifold groups (see \cite[Theorem 2-(6)]{KLR16}). Armed with this knowledge, it is natural to try to use the Farrell-Jones isomorphism conjecture and the aforementioned models for $\underline{\underline{E}}\Gamma$ to compute the algebraic K-theory groups of $\dbZ \Gamma$ and the Whitehead groups of $\Gamma$. The main goal of this paper is to carry out this task.

In this paper, we provide formulas for the Whitehead groups $Wh^R_{i}(R[\Gamma])$, a direct factor of $K_i(R[\Gamma])$ where $\Gamma$ is a $3$-manifold group possibly with torsion and  $R$ is any ring with unitary element. As a byproduct we describe a summand of $K_i(R[\Gamma])$, explicitly, the summand given by the relative homology groups $H^{\Gamma}_*(\underline{\underline{E}}\Gamma,\underline{E}\Gamma; \dbK_R)$, see Section \ref{prel} for their definition.  Our main tools to carry out with these computations are: the $K$-theoretic Farrell-Jones isomorphism conjecture, the Knesser-Milnor and JSJ-de\-compo\-sitions, and the models for $\evc \Gamma$ constructed in \cite{JLSS}. 

We now outline more explicitly our strategy and state the main results of this paper. Let $\Gamma$ be a 3-manifold group, and let $\evc \Gamma$, $\efin \Gamma$ be models for the universal space with virtually cyclic isotropy and finite isotropy, respectively. By the veracity of the Farrell-Jones isomorphism conjecture for $\Gamma$ we have isomorphisms for $K$-theory and Whitehead groups
 \begin{align*}
K_*(R[\Gamma])\cong H^\Gamma_*(\underline{E}\Gamma; \dbK_R)\oplus H^\Gamma_*(\underline{\underline{E}}\Gamma,\underline{E}\Gamma; \dbK_R)\\
Wh^R_*(\Gamma) \cong H^\Gamma_*(\underline{E}\Gamma, E\Gamma; \dbK_R)\oplus H^\Gamma_*(\underline{\underline{E}}\Gamma,\underline{E}\Gamma; \dbK_R).
\end{align*}

 The main purpose of this work is to describe the relative terms $H^{\Gamma}_*(\underline{E}\Gamma, E\Gamma; \dbK_R)$ and $H^\Gamma_*(\underline{\underline{E}}\Gamma,\underline{E}\Gamma; \dbK_R)$. Our main theorems are as follows.
 
 In the following theorem we describe the first direct factor of $Wh_i^R(\Gamma)$.
 
 \begin{relfinite}[Relative terms for finites,  \cref{finiterel}]
Let $M$ be a closed, oriented, connected  3-manifold, and let $\Gamma$ be the fundamental group of $M$.  Let \[M=Q_1\#\cdots\#
Q_n\#P_1\#\cdots\#
P_m\] be the prime decomposition of $M$ such that  $Q_i$, $1\leq j \leq n$, are exactly those manifolds (if any) in the decomposition that are spherical. Denote $\Gamma_j=\pi_1(Q_j)$. Then, for all $i\in \dbZ$
\[
H^{\Gamma}_i(\underline{E}\Gamma, E\Gamma; \dbK_R)\cong \bigoplus_{j=1}^{n} Wh_i^R(\Gamma_j). 
\]
Equivalently, the splitting of $H^{\Gamma}_i(\underline{E}\Gamma, E\Gamma; \dbK_R)$ as a direct sum runs over the conjugacy classes of maximal finite subgroups of $\Gamma$.
In particular the \emph{classical} Whitehead group $Wh(\Gamma)=Wh_1(\Gamma)$ of $\Gamma$ is always a finitely generated abelian group.
\end{relfinite}

As for the second relative term $H^\Gamma_*(\underline{\underline{E}}\Gamma, \underline{E}\Gamma)$ of $Wh_i^R(\Gamma)$ we proceed en several steps.
First, assuming $\Gamma=\pi_1(M)$ is any 3-manifold group, we describe $H^\Gamma_*(\underline{\underline{E}}\Gamma, \underline{E}\Gamma)$ by means of a long exact sequence in terms of the corresponding relative terms of the fundamental groups of the prime factors $M$, and certain Nil-groups. This is done in the following theorem.


\begin{relvc}[Relative terms for virtually cyclics: prime case \cref{reltermprimedecomposition}]
Let $\Gamma=\pi_1(M)$ be a $3$-manifold group. Consider the prime decomposition $M=P_1\#\cdots \# P_m$ and the corresponding splitting $\Gamma\cong \Gamma_1 *\cdots * \Gamma_m$, where $\Gamma_i=\pi_1(P_i)$. Then the relative term $H^\Gamma_*(\underline{\underline{E}}\Gamma, \underline{E}\Gamma)$ fits in the long exact sequence 
\begin{align*}
 \cdots \to  \bigoplus_{i=1}^m H_n^{\Gamma_i}(\underline{\underline{E}} \Gamma_i, \underline{E}\Gamma_i)
 &\to H_n^\Gamma(\underline{\underline{E}}\Gamma, \underline{E}\Gamma) \to\\
 &\to \bigoplus_{\calH_1}2NK_n(R)\oplus \bigoplus_{\calH_2}NK_n(R) \to  \bigoplus_{i=1}^m H_{n-1}^{\Gamma_i}(\underline{\underline{E}} \Gamma_i, \underline{E}\Gamma_i)
   \to \cdots
 \end{align*}
where $\calH_1$ (resp. $\calH_2$) is a set of representatives of $\Gamma$-conjugacy classes of maximal elements in $\vcyc\setminus\calV$ that are isomorphic to $\dbZ$ (resp. $D_\infty$), and $\calV$ is the family of virtually cyclic subgroups of $\Gamma$ that are subconjugated to some $\Gamma_i$. 
\end{relvc}

The next natural step is to describe the term $H^\Gamma_*(\underline{\underline{E}}\Gamma, \underline{E}\Gamma)$ when $\Gamma$ is the fundamental group of a prime 3-manifold $M$. The following theorem provides such a description, save for two exceptional cases, again by means of two short exact sequences. Here the relative term $H^\Gamma_*(\underline{\underline{E}}\Gamma, \underline{E}\Gamma)$ is described using the corresponding relative terms of the JSJ-pieces of $N$ (each of these pieces is either a Seifert fibered manifold or a hyperbolic manifold), and certain Nil-groups.

\begin{relvc}[Relative terms for virtually cyclics: the JSJ-case \cref{reltermjsjdecomposition}] Let $N$ be a prime manifold with fundamental group $\Gamma$. Assume that the minimal JSJ-decomposition of $N$ is not a double of a twisted $I$-bundle over the Klein bottle.
Then the relative term $H^{\Gamma}_*(\underline{\underline{E}}\Gamma, \underline{E}\Gamma)$ fits in the long exact sequence 
 \begin{align*}
 \cdots \to  H_n^{\Gamma}(E_\calW \Gamma, \underline{E}\Gamma)
 &\to H_n^{\Gamma}(\underline{\underline{E}}\Gamma, \underline{E}\Gamma) \to\\
 &\to \bigoplus_{\calH_1}2NK_n(R)\oplus \bigoplus_{\calH_2}NK_n(R) \to  H_{n-1}^{\Gamma}(E_\calW \Gamma, \underline{E}\Gamma)
   \to \cdots
 \end{align*}
where $\calH_1$ (resp. $\calH_2$) is a set of representatives of $\Gamma$-conjugacy classes of maximal elements in $\vcyc\setminus\calW$ that are isomorphic to $\dbZ$ (resp. $D_\infty$), and $\calW$ is the family of virtually cyclic subgroups of $\Gamma$ that are subconjugated to some vertex group of a suitable tree $\mathbf{X}$. 
 Moreover, the term  $H_n^{\Gamma}(E_\calW \Gamma, \underline{E}\Gamma)$ fits in the long exact sequence
\begin{align*}
\cdots \to \bigoplus_{E(S)}\bigoplus_{1}^\infty (2NK_n(R) \oplus 2NK_{n-1}(R))
 &\to \bigoplus_{ i=1}^m H_n^{\Gamma_i}(\underline{\underline{E}}\Gamma_i, \underline{E}\Gamma_i)\\
&\to H_n^{\Gamma}(E_\calW \Gamma,\underline{E}\Gamma)\to \cdots \end{align*}
where $E(S)$ is the edge set of $\mathbf{X}$.
\end{relvc}

The third step is to describe the relative term $H^\Gamma_*(\underline{\underline{E}}\Gamma, \underline{E}\Gamma)$ when $\Gamma$ is the fundamental group of a Seifert fibered manifold, a hyperbolic 3-manifold or we are in the exceptional cases of the previous theorem. We make a summary of references for all the results we obtained in these cases:

\begin{table}[h!]
\begin{tabular}{|l|l|}
\hline
\textbf{Type of manifold}                                                     & \textbf{Analyzed in}  \\ \hline
Exceptional case (modeled on $\mathrm{Sol}$)                                                  & \cref{section:exceptional:cases}            \\\hline
Hyperbolic manifold                                                  &  \cref{thm:rel:term:hyperbolic}           \\ \hline
Seifert fibered manifold with orbifold base modeled on $\mathbb E^2$ &      \cref{section:Sefiert:spherical}       \\ \hline
Seifert fibered manifold with orbifold base modeled on $\mathbb H^2$ &  \cref{section:seifert:flat}           \\ \hline
Seifert fibered manifold with orbifold base modeled on $\mathbb S^2$ &  \cref{section:seifert:hyperbolic}           \\ \hline
\end{tabular}
\end{table}
 
The paper is organized as follows, we recall the ingredients to establish the Farrell-Jones conjecture for $K$-theory, some tools to approach the relative terms and  decompositions of 3-manifolds in section \ref{prel}. Next, in section \ref{vc3} we classify up to isomorphism the virtually cyclic subgroups in 3-manifold groups. In section \ref{reltermf}, we describe the long exact sequence 
that concerns the relative term for finite groups and in section \ref{RelTerm}, we establish the long exact sequences that concern
the relative term for infinite virtually cyclic groups, here we have to consider the prime decomposition and the JSJ-decomposition 
differently. Lastly, we describe in sections \ref{reltermseifert} and \ref{reltermhyp}, the relative terms for Seifert 3-manifolds and hyperbolic manifolds respectively.

\section*{Acknowledgements}
This work originated while the first Author was on leave in the fall 2019 at the Normandie Univ. UNICAEN, CNRS, Laboratoire de Math\'ematiques
Nicolas Oresme UMR CNRS 6139, the first Author acknowledges the hospitality and support by the CNRS. The first Author was also supported by grants CB-CONACYT-283988 and UNAM-DGAPA-PAPIIT-IN1055318. The second author thanks the hospitality of the Centro de Ciencias Matemáticas, UNAM where part of this job was written. Both authors thank to Mauricio Bustamante for comments in a draft of the present paper.


\section{Preliminaries}\label{prel}

\subsection{Classifying spaces for families of subgroups}

Given a group $G$, we say that a collection of subgroups $\calF$ is a \emph{family} if it is closed under conjugation and under taking subgroups. We say that a $G$-CW-complex $X$ is a model for the classifying space $E_\calF G$ if every isotropy group of $X$ belongs to $\calF$, and $X^H$ is contractible whenever $H$ belongs to $\calF$. Such a model always exists and it is unique up to $G$-homotopy equivalence.

Let $\calG$ be a second family of subgroups of $G$ such that $\calF\subseteq \calG$. Then we have a cellular $G$-map $E_\calF G \to E_\calG G$ that is unique up to $G$-homotopy. Without loss of generality we may assume that this map is an inclusion, since the mapping cylinder of $E_\calF G \to E_\calG G$ is again a model for $E_\calG G$, and thus we can replace this map by the inclusion of $E_\calF G$ in the new model for $E_\calG G$. Thus the pair $(E_\calG G, E_\calF G)$ is well-defined.

In the present work we are mainly concerned with the family, $\vcyc$, of virtually cyclic subgroups. A related family is the family, $\fin$, of finite subgroups. We will denote $E_{\vcyc} G$   and $E_{\fin} G$ as $\evc G$  and $\efin G$,  respectively. 

\subsection{The Farrell-Jones isomorphism conjecture}

Let $G$ be a discrete group and let $R$ be an associative ring with unit. We denote by $K_n(R[G])$, $n\in \dbZ$, the algebraic $K$-theory groups of the group ring $R[G]$ in the sense of Quillen for $n\geq0$ and in the sense of Bass for $n\leq -1$. Let $NK_n(R)$ denote the Bass Nil-groups of $R$, which by definition, are the cokernels of the maps in algebraic $K$-theory $K_n(R)\to K_n(R[t])$ induced by the canonical inclusion $R\to R[t]$. From the Bass-Heller-Swan theorem  \cite[7.4]{Bass} we get, for all $n\in\dbZ$, the decomposition $$K_n(R[\dbZ])\cong K_n(R[t,t^{-1}])\cong K_n(R)\oplus K_{n-1}(R) \oplus NK_n(R)\oplus NK_n(R).$$
From now on we will denote the sum $NK_n(R)\oplus NK_n(R)$ by $2NK_n(R)$, and in general, the $m$-fold sum of copies of $NK_n(R)$ by $mNK_n(R)$.


Throughout this work we consider equivariant homology theories in the sense of \cite[Section 2.7.1]{LR05}. In particular, we are interested in the equivariant homology theory with coefficients in the $K$-theory spectrum described in \cite[Section 2.7.3]{LR05}, denoted by $H^G_*(-;\KR)$. For a fixed group $G$ this homology theory satisfies the Eilenberg-Steenrod axioms in the $G$-equivariant setting. One of the main properties of this homology theory is that 
\[ 
H_n^G(G/H;\KR) \cong H_n^H(H/H;\KR) \cong K_n(R[H])
\]
for every $H\leq G$ and for all $n\in \dbZ$. 
This equivariant homology theory is relevant since it appears in the statement of the Farrell-Jones isomorphism conjecture.\\

In their seminal paper \cite{FJ93} Farrell and Jones formulated their
fundamental isomorphism conjecture for the $K$-theory, 
$L$-theory and Pseudoisotopy functors. Here we consider the $K$-theoretic version of the conjecture as stated by Davis and L\"uck in \cite{DL98}.

\begin{conjecture}[The Farrell-Jones isomorphism conjecture] Let $G$ be group and let $R$ be a ring. Then, for any $n\in \dbZ$, the following assembly map, induced by the projection $\underline{\underline{E}}G\to G/G$, is an isomorphism

\beq\label{FJ}\tag{$\ast$}
\func{A_{\vcyc,\all}}{H^{G}_n(\underline{\underline{E}}G;\KR)}{H^{G}_n(G/G;\KR)\cong K_n(R[G])}.
\eeq
\end{conjecture}

The class of groups for which this conjecture is valid is substancial, a list may be found in \cite[Theorem 2-(6)]{KLR16}. In this paper we use the fact that the Farrell-Jones isomorphism conjecture is true for 3-manifold groups. For completeness we state the following theorem.

\begin{theorem}\cite[Theorem 2-(6)]{KLR16}\label{FJvalid}
Let $\Gamma$ be a 3-manifold group. Then $\Gamma$ satisfies the Farrell-Jones isomorphism conjecture.
\end{theorem}

Once the Farrell-Jones conjecture has been verified for a group $G$, one can
hope to compute $K_n(R[G])$ by computing the left hand side of \eqref{FJ}.
This is a generalized homology theory that can be approached, for 
example, via Mayer-Vietoris sequences, Atiyah-Hirzebruch-type spectral
sequences or the $p$-chain spectral sequence described in \cite{DL03}.


\subsection{Whitehead groups}

Let $G$ be a group that satisfies the Farrell-Jones conjecture. From \cite[Prop. 15.7]{Wa78}  we have, for all $n\in\dbZ$, the following isomorphism that we take as definition
\[Wh^R_n(G)\cong H_n^G(\evc G,EG;\dbK_R).\]
In fact, the long exact sequence of the pair $(\underline{\underline{E}}G, EG)$ yields the long exact sequence
\begin{eqnarray*}
\cdots\to & H_n(BG;\dbK_R(G/1)) \to K_n(R [G])\to Wh^R_n(G)\to \\
\to &  H_{n-1}(BG;\dbK_R(G/1))\to K_{n-1}(R[G])\to \cdots
\end{eqnarray*}
where $H_n(BG;\dbK_R(G/1))$ is the classical generalized homology theory with coefficients in the spectrum $\dbK_R(G/1)$ which has as homotopy groups the algebraic \ktheory of the ring $R$.

{\color{red}

}


\subsection{Computations of $K$-theory and Whitehead groups}

By the main theorem of  \cite{Ba03}, the inclusion $\underline{E}G \to \underline{\underline{E}}G$ induces a split injection $H^G_*(\underline{E}G, \dbK_R)\to H^G_*(\underline{\underline{E}}G, \dbK_R)$. Thus we have the following splitting
\[
 H^G_*(\underline{\underline{E}}G; \dbK_R)\cong H^G_*(\underline{E}G; \dbK_R)\oplus H^G_*(\underline{\underline{E}}G,\underline{E}G; \dbK_R).
\]
If additionally $G$ satisfies the Farrell-Jones conjecture, we get the following isomorphisms
\[
K_*(R[G])\cong H^G_*(\underline{E}G; \dbK_R)\oplus H^G_*(\underline{\underline{E}}G,\underline{E}G; \dbK_R)
\]
and
\[
Wh^R_*(G) \cong H^G_*(\underline{E}G, EG; \dbK_R)\oplus H^G_*(\underline{\underline{E}}G,\underline{E}G; \dbK_R)
\]
see for instance \cite[Lemma~3.4]{SSV18}.

In this work we  give descriptions of $H^G_*(\underline{E}G, EG; \dbK_R)$ and $H^G_*(\evc{G}, \efin{G}; \dbK_R)$ for $G$ a $3$-manifold group, hence, by Theorem \ref{FJvalid}, of $Wh_*^R(G)$. In the first case, we prove that a $3$-manifold group satisfies properties (M) and (NM), so that we can run verbatim the proof of the main theorem of \cite{BSS}.

We  then  analyse $H^G_*(\underline{\underline{E}}G,\underline{E}G; \dbK_R)$ in Section \ref{RelTerm}. In order to achieve this, we will use the models for $\underline{\underline{E}}G$, for $3$-manifold groups, constructed in \cite{JLSS}.


The spectral sequence we are about to deduce might be well known to the experts, but we include the details due to the lack of a suitable reference. 

\begin{theorem}\label{cor:whitehead:ss}
Let $\calF\subseteq \calG$ be families of subgroups of $G$. Then there is a \emph{relative} Atiyah-Hirzebruch-type spectral sequence that converges to 
$$
H^G_*(E_\calG G,E_\calF G;\dbK_R),
$$
such that the second page is given by

\[
E^2_{p,q}=H_p(B_\calG G; \{ H_q^{G_{\sigma^p}}(pt,E_{\calF \cap G_{\sigma^p}}G_{\sigma^p};\dbK_R) \})
\]
where the right hand side is homology with local coefficients, and $\sigma^p$ is a $p$-cell of $B_\calG G$.
In particular, if $\calF$ is the trivial family, we have
\[
E^2_{p,q}=H_p(B_\calG G; \{ Wh_q^R(G_{\sigma^p}) \}).
\]
\end{theorem}
\begin{proof}

Denote $\dbK=\dbK_R$. Following \cite[Theorem~4.1]{DQR11}, there exists an $Or(G)$-spectrum $\dbK_\calF$, and a homotopy cofiber sequence of $Or(G)$-spectra
\beq\label{eq:K:cof:ses}
\dbK_\calF \to \dbK \to \dbK/ \dbK_\calF
\eeq
such that $H_*^G(X;\dbK_\calF)\cong H_*^G(X\times E_\calF G;\dbK)$, for every $G$-space $X$.
Since the product of models for $E_\calG G$ and $E_\calF G$ is a model for $E_\calF G$ 
\beq\label{eq:product:classif:spaces}
H_*^G(E_\calG G;\dbK_\calF)\cong H_*^G(E_\calG G\times E_\calF G;\dbK)\cong H_*^G(E_\calF G;\dbK).
\eeq
The long exact sequence associated to \eqref{eq:K:cof:ses} and $E_\calG G$, and the isomorphism \eqref{eq:product:classif:spaces}, yields the following long exact sequence
\[
\cdots \to H_n^G(E_\calF G;\dbK) \to H_n^G(E_\calG G;\dbK) \to H_n^G(E_\calG G;\dbK/\dbK_\calF)\to\cdots .
\]
Therefore
\[H_n^G(E_\calG G;\dbK/\dbK_\calF)\cong H_n^G(E_\calG G,E_\calF G;\dbK).
\]
Now  from \cite[Theorem~4.7]{DL98} applied to $E_\calG G$, and a model for $\dbK/\dbK_\calF$ together with its skeletal filtration, we obtain a spectral sequence that converges to $H^G_*(E_{\calG}G,E_\calF G;\dbK_R)$, such that the second page is given by

\[
E^2_{p,q}=H_p(B_\calG G; \{ H^G_q(G/G_{\sigma^p} ;\dbK_R/\dbK_\calF)\})
\]
where the right hand side is homology with local coefficients, and $\sigma^p$ is a $p$-cell of $B_\calG G$.

It remains to prove that
$$
H_q^G(G/H;\dbK_R/\dbK_\calF)\cong H_q^{G_{\sigma^p}}(pt,E_{\calF \cap H}H)
$$
for every subgroup $H$ of $G$.
From \eqref{eq:K:cof:ses} we have the following long exact sequence 
\[
\cdots \to H_q^G(G/H;\dbK_\calF) \to H_q^G(G/H;\dbK) \to H_q^G(G/H;\dbK/\dbK_\calF)\to\cdots .
\]
By \cite[Theorem~4.1(i)]{DQR11} the homomorphism 
$$
H_q^G(G/H;\dbK_\calF) \to H_q^G(G/H;\dbK)
$$
can be identified with the homomorphism 
$$
H_q^H(E_{\calF \cap H}H;\dbK)\to H^H_q(pt;\dbK).
$$ 
Therefore, 
\[H_q^G(G/H;\dbK/\dbK_\calF)\cong H^H_q(pt,E_{\calF \cap H}H;\dbK).\]
\end{proof}

As an immediate application of the spectral sequence obtained in \cref{cor:whitehead:ss} we have the following result that will be useful later.

\begin{corollary}\label{corollary:kles:relative:terms}
Let $\calF\subseteq\calG$ be families of subgroups of $G$. Assume that there is a one-dimensional model for $E_\calG G$. Then we have the following long exact sequence 
\begin{align*}
 \cdots \to \bigoplus_{e\in E} H_q^{G_{e}}(pt,E_{\calF \cap G_{e}}G_{e};\dbK_R) &\to \bigoplus_{v\in V} H_q^{G_{v}}(pt,E_{\calF \cap G_{v}}G_{v};\dbK_R) \to H_q^G(E_\calG G, E_{\calF}G)\to \cdots .
 \end{align*}
\end{corollary}
\subsection{Relative terms and Nil-groups}
All the material in this section can be found, for instance, in \cite{DQR11} and \cite{DKR11}.

\subsubsection{Groups that surject onto $\dbZ$}
Let $G$ be a group that surjects to $\dbZ$ with kernel $N$, i.e. $G$ fits in the short exact sequence 
\[1\to N\to G\to \dbZ \to 1.\]
Hence $G$ is isomorphic to the semi-direct product $N\rtimes_\phi \dbZ$ with $\phi$ an automorphism of $N$, and the group ring $R[N\rtimes_\phi \dbZ]$ can be identified with the twisted Laurent polynomial ring $R[N]_\phi [t,t^{-1}]$. As in the untwisted case, we can define the Farrell-Hsiang Nil groups $NK(R[N];\phi)$ and we get the following Bass-Heller-Swan type theorem:
\[
K_n(R[G])\cong K_n(R[N])\oplus K_{n-1}(R[N]) \oplus NK_n(R[N];\phi)\oplus NK_n(R[N];\phi),
\]
and the Whitehead groups version
\[Wh^R_n(G)\cong Wh^R_n(N)\oplus Wh^R_{n-1}(N) \oplus NK_n(R[N];\phi)\oplus NK_n(R[N];\phi).\]

\subsubsection{Groups that surject onto $D_\infty$}

Let $G$ be a group that surjects to $D_\infty$ with kernel $N$, i.e. $G$ fits in the short exact sequence 
\[1\to N\to G\to D_\infty \to 1 .\]
Hence $G$ is isomorphic to the amalgamated product $G_1\ast_N G_2$, where $G_1$ and $G_2$ are the pre-images under the surjection above of the $\dbZ/2$-factors in the splitting $D_\infty=\dbZ/2\ast \dbZ/2$.

There exist \emph{certain} groups, called Waldhausen Nil-groups, denoted as
\[
NK_n(R[N]; R[G_1-N],R[G_2-N])
\] 
such that it is a summand of $K_n(R[G])$, and we have the following Mayer-Vietoris type long exact sequence
\begin{align*}
\cdots \to K_n(R[N])&\to K_n(R[G_1])\otimes K_n(R[G_2])\to\\ & K_n(R[G])/NK_n(R[N]; R[G_1-N],R[G_2-N])\to \cdots   . 
\end{align*}
We also  have a version that involves the Whitehead groups of $G$:
\begin{align*}
\cdots \to Wh^R_n(N)&\to Wh^R_n(G_1)\otimes Wh^R_n(G_2)\to\\ & Wh^R_n(G)/NK_n(R[N]; R[G_1-N],R[G_2-N])\to \cdots  .  
\end{align*}
On the other hand, the infinite dihedral group $D_\infty$ has an index 2 subgroup isomomorphic to $\dbZ$, and therefore, $G$ has an index two subgroup $\barG$ isomorphic to $N\rtimes_\phi \dbZ$. A remarkable theorem of \cite{DKR11} and \cite{DQR11} states the existence of an isomorphism

\[
NK_n(R[N]; R[G_1-N],R[G_2-N])\cong NK_n(R[N];\phi)
\]
that is, the Waldhausen Nil-groups of $G$ are isomorphic to the Farrell-Hsiang Nil-groups of $\barG$.

\subsubsection{Relative terms}\label{subsubsectio:relative:terms}

With the notation above, let $\calF$ be the smallest family of $G$ containing $G_1$ and $G_2$, and let $\calF_0$ be the smallest family of $G$ containig $N$. Assume that $G$ satisfies the Farrell-Jones conjecture, then we have the following isomorphisms:
\begin{align*}
    Wh_n^R(G)&\cong H_n^G(E_\calF G, EG)\oplus H_n^G(\underline{\underline{E}}G,E_\calF G) \\
    Wh_n^R(\barG)&\cong H_n^{\barG}(E_{\calF_0} \barG, E\barG)\oplus H_n^{\barG}(\underline{\underline{E}}\barG,E_\calF \barG)
\end{align*}
and
\begin{align*}
    H_n^G(\underline{\underline{E}}G,E_\calF G)&\cong NK_n(R[N]; R[G_1-N],R[G_2-N])\cong NK_n(R[N];\phi)\\
    H_n^{\barG}(\underline{\underline{E}}\barG,E_{\calF_0} \barG)&\cong 2NK_n(R[N];\phi).
\end{align*}
Moreover, $H_n^G(E_\calF G, EG)$ fits in the following Mayer-Vietoris type long exact sequence
\[
\cdots \to Wh^R_n(N)\to Wh^R_n(G_1)\otimes Wh^R_n(G_2)\to H_n^G(E_\calF G, EG)\to \cdots\]
and we have the following isomorphism
\[H_n^{\barG}(E_{\calF_0} \barG, E\barG)\cong Wh^R_n(N)\oplus Wh^R_{n-1}(N). \]
In particular, if $N$ is a finite group, that is, when $G$ and $\barG$ are virtually cyclic, we get the following isomorphisms
\begin{align*}
    &H_n^G(\underline{\underline{E}}G,E_\calF G)\cong H_n^G(\underline{\underline{E}}G,\underline{E} G) \cong NK_n(R[N]; R[G_1-N],R[G_2-N])\cong NK_n(R[N];\phi)\\
    &\text{ and }\\
    &H_n^{\barG}(\underline{\underline{E}}\barG,E_{\calF_0} \barG)\cong H_n^{\barG}(\underline{\underline{E}}\barG,\underline{E} \barG) \cong 2NK_n(R[N];\phi).
\end{align*}


\subsection{Prime and JSJ decomposition of a $3$-ma\-ni\-fold}
A \emph{closed} $3$-ma\-ni\-fold is a $3$-ma\-ni\-fold that is compact with empty
boundary.  A \emph{connected sum} of two $3$-manifolds $M$ and $N$, denoted $M\#
N$, is a manifold created by removing the interiors of a smooth $3$-disc $D^3$
from each manifold, then identifying the boundaries $\dbS^{2}$.  A $3$-manifold is
\emph{nontrivial} if it is not homeomorphic to $\dbS^3$.  A nontrivial $3$-manifold, $M$, is \emph{prime}
if it  cannot be decomposed as a connected
sum of two nontrivial $3$-manifolds; that is, $M=N\# P$ for some $3$-manifolds
$N,P$ forces either $N=\dbS^3$ or $P=\dbS^3$.  A $3$-manifold $M$ is called
\emph{irreducible} if every embedded 2-sphere $\dbS^2\subset M$ bounds a ball $D^3\subset M$. It is well-known that all orientable prime $3$-manifolds are irreducible with the exception of $S^1\times
S^2$. The following is a well-known theorem of Kneser (existence) and Milnor (uniqueness) \cite{AFW15}[1.2.1].

\begin{theorem}[Prime decomposition]
\label{prime decomposition}
Let $M$ be a connected, closed, oriented  3-manifold.  Then $M=P_1\#\cdots\#
P_n$ where each $P_i$ is prime.  Furthermore, this decomposition is unique up to
order and homeomorphism.
\end{theorem}

Another well known result we will need is the Jaco--Shalen--Jo\-hann\-son decomposition, after Perelman's work \cite{AFW15}[1.6.1].

\begin{theorem}[JSJ decomposition after Perelman's theorem]
\label{jsj decomposition}
For a closed, prime, connected, oriented 3-manifold $N$ there exists a collection (possibly empty)
$T\subseteq N$ of disjoint incompressible tori, i.e. two sided properly embedded and $\pi_1$-injective, such that each component of
$N\setminus T$ is either a hyperbolic or a Seifert fibered (noncompact) manifold.  A minimal
such collection $T$ is unique up to isotopy.
\end{theorem}

If the collection of tori provided by \Cref{jsj decomposition} is empty, we will say that the JSJ-decomposition of $M$ is \emph{trivial}, otherwise we will say the JSJ-decomposition of $M$ is \emph{nontrivial}.

\begin{remark}\label{prime:geometric:splittings}
Note that the prime decomposition provides a graph of groups with trivial edge groups and vertex groups 
isomorphic to the fundamental groups of the $P_i$'s, see \cite{Se03}. The fundamental group of the graph of groups will be isomorphic to $\pi_1(M)$. Similarly the JSJ decomposition of a prime $3$-manifold $N$ gives rise to 
a graph of groups, with all edge groups isomorphic to $\dbZ^2$, and vertex groups isomorphic to the 
fundamental groups of the Seifert fibered and hyperbolic pieces. Again, the fundamental group of the 
graph of groups will be isomorphic to $\pi_1(N)$. Each graph of groups provide a splitting for the 
fundamental groups of the initial manifold. These splittings will be used to analyse the relative terms in the following sections.
\end{remark}

We use the following notation for Thurston's eight geometries: $\dbE^3$ (flat geometry), $\dbS^3$ (spherical geometry), $\dbH^3$ (hyperbolic geometry), $\dbS^2\times \dbE$, $\dbH^2\times \dbE$, $\widetilde{PSL}_2(\dbR)$, $\mathrm{Nil}$, and $\mathrm{Sol}$.
We finish this section with the following theorem that will be useful later.

\begin{theorem}\label{EvansMoserSolvable}\cite{EM72}
Let $G$ be a  (virtually) solvable infinite 3-manifold group. Then $G$ appears in the following list of groups:
\begin{itemize}
\item $\dbZ$, $\dbZ^2$, or $\calK=\dbZ\rtimes\dbZ$.

\item An extension $1\to \dbZ^2\to G\to \dbZ \to 1$, i.e., $G$ is isomorphic to a semi-direct product $\dbZ^2\rtimes_\phi \dbZ$ with $\phi$ an automorphism of $\dbZ^2$.

\item A free product of the form $\calK \ast_{\dbZ^2} \calK$, where the $\dbZ^2$ embeds in each $\calK$ as a subgroup of index 2. In particular, we have a short exact sequence $1\to \dbZ^2 \to G\to D_\infty \to 1$.
\end{itemize}
\end{theorem}


\section{Classification of virtually cyclic subgroups}\label{vc3}

In this section we classify all virtually cyclic subgroups of a $3$-manifold group $\Gamma$.
First we proceed to classify all finite subgroups of a $3$-manifold group.

\begin{lemma}\label{lemma:JSJ:torsionfree}
Let $N$ be a connected, closed, oriented, prime $3$-manifold, and let $\Gamma$ be its fundamental group. If the JSJ-decomposition of $N$ is non-trivial, then  $\Gamma$ is torsion free.
\end{lemma}
\begin{proof}
If $N$ has a non-trivial JSJ-decomposition, then every JSJ-piece is a noncompact Seifert fibered manifold  or noncompact hyperbolic. 
Since every (noncompact) hyperbolic manifold is aspherical, we have that its fundamental group is torsion free.

On the other hand, every manifold covered by $\dbS^2\times \dbE$, has fundamental group isomorphic to a subgroup of either $\dbZ$ or $D_\infty$ \cite[Table~1]{AFW15}, and every manifold covered by $\dbS^3$ has finite fundamental group. Thus these manifolds cannot appear as pieces of a JSJ decomposition since they cannot contain any copy of $\dbZ^2$. Then by Corollary~1.2.1 and Theorem~1.2.2 from \cite{Mo05}  a  Seifert fibered JSJ-piece is covered by the contractible spaces $\dbH^2\times \dbE$, $\widetilde{PSL}_2(\dbR)$, $\mathrm{Sol}$, or $\mathrm{Nil}$, or it is homeomorphic to $T^2\times I$, the twisted $I$-bundle over the Klein bottle, or the solid torus. 
Therefore every piece of the JSJ-decomposition of $N$ has torsion free fundamental group. By \Cref{prime:geometric:splittings}, $\Gamma$ is isomorphic to the fundamental group of a graph of groups with torsion free vertex groups. Let $F$ be a finite subgroup of $\Gamma$, then by a standard argument, $F$ fixes a vertex in the Bass-Serre tree of $\Gamma$. Thus $F$ is subconjugated to a vertex group, and in consequence, $F$ is trivial. Therefore the fundamental group of $N$ is torsion free. \end{proof}

\begin{proposition}\label{prop:classify:torsion}
Let $\Gamma$ be a $3$-manifold group, and let $F$ be a finite subgroup of $\Gamma$. Then $F$  is either cyclic, or a central extension of a dihedral, tetrahedral, octahedral, or icosahedral group by a cyclic group of even order. Moreover, any finite subgroup of $\Gamma$ is subcojugated to the fundamental  group of a prime manifold (from the prime decomposition) that is either covered by $\dbS^3$ of $\dbS^2\times \dbE$.
\end{proposition}
\begin{proof}
Let $M$ be a $3$-manifold such that $\Gamma=\pi_1(M)$. 
From the prime decomposition of $M=P_1\#\cdots\#
P_n$ we get a splitting $\Gamma=\Gamma_1*\cdots * \Gamma_n$, where $\Gamma_i=\pi_1(P_i)$. If $F$ is a finite subgroup of $\Gamma$, then, by a standard result in Bass-Serre theory, $F$ has to be subconjugated to one of the $\Gamma_i$'s. Hence we only have to classify finite subgroups of prime $3$-manifolds.

Let $N$ be a prime $3$-manifold. We have two cases depending on whether the JSJ-decomposition of $N$ is trivial or not. 

Assume the JSJ-decomposition of $N$ is trivial, i.e. $N$ is either Seifert fibered or hyperbolic. Any hyperbolic manifold has torsion free fundamental group since it is aspherical. If $N$ is Seifert fibered, then the only possibility for $N$ to have torsion would be that $N$ is covered by $\dbS^3$ or $\dbS^2\times \dbE$, otherwise $N$ would be aspherical \cite[Theorem~1.2.2]{Mo05}. In this case $F$ can only be one of the following possibilities: \begin{itemize}
    \item If $N$ is covered by $\dbS^2\times \dbE$, then the fundamental group of $N$ is isomorphic to either $\dbZ$ or $D_\infty$ \cite[Table~1]{AFW15}. Hence $F$ is either trivial or isomorphic to $\dbZ/2$.
    \item If $N$ is covered by $\dbS^3$, then the fundamental group of $N$ is either cyclic, or a central extension of a dihedral, tetrahedral, octahedral, or icosahedral group by a cyclic group of even order (see \cite[Section~1.7]{AFW15}). Thus $F$ also is one of the previously mentioned possibilities.
\end{itemize}

On the other hand if $N$ has non-trivial JSJ-decomposition, then by \Cref{lemma:JSJ:torsionfree}, the  fundamental group of $N$ is torsion free. Hence $F$ is trivial in this case.
\end{proof}

Next, we proceed to classify the infinite virtually cyclic subgroups of a $3$-manifold group. Recall that from \cite{JPL06}  every virtually cyclic subgroup fits in one of the following categories
\begin{itemize}
    \item finite, or
    \item isomorphic to an amalgamated product $F_1*_{F_2} F_3$, where each $F_i$ is finite and $F_2$  has index 2 in both $F_1$ and $F_3$, or
    \item isomorphic to a simedirect product $F\rtimes \dbZ$, where $F$ is a finite group.
\end{itemize}

\begin{proposition}\label{prop:classify:vcyc}
Let $\Gamma$ be a $3$-manifold group, and let $V$ be an infinite virtually cyclic subgroup of $\Gamma$. Then $V$ is either isomorphic to  $\dbZ$ or $D_\infty$.
\end{proposition}

\begin{proof}
Let $M$ be a $3$-manifold such that $\Gamma=\pi_1(M)$. Let $V$ be an infinite virtually cyclic subgroup of $\Gamma$.
From the prime decomposition of $M=P_1\#\cdots\#
P_n$ we get a splitting $\Gamma=\Gamma_1*\cdots * \Gamma_n$, where $\Gamma_i=\pi_1(P_i)$. Let $T$ be the Bass-Serre tree of the splitting. Then we have two mutually exclusive possibilities: $V$ either fixes a point of $T$ or it acts nontrivially on a geodesic line (see \cite{JPLMP11}). 

Assume $V$ acts nontrivally on a geodesic $\gamma$. Then $V$ is either isomorphic to $\dbZ$ or to $D_\infty$. In fact, $V$ fits in a short exact sequence
\[
1\to K\to V\to D\to 1
\]
where $D$ is a finite subgroup of $D_\infty$, and $K$ is the subgroup of all elements of $V$ that act trivially in $\gamma$. Since the edge stabilizers of $T$ are trivial, then $K$ is trivial. Thus $V$ embeds into $D_\infty$. Now the assertion follows since every subgroup of $D_\infty$ is isomorphic to either $\dbZ$ or $D_\infty$.

Assume that $V$ fixes a vertex of $T$, then $V$ is subconjugated to $\Gamma_i=\pi_1(P_i)$, for some $1\leq i \leq n $. We have two cases: either $P_i$ has trivial or non-trivial JSJ-decomposition. 

In the first case $P_i$ is either hyperbolic or Seifert fibered. Hence from \cite[Theorem~1.2.2]{Mo05} $P_i$ is aspherical, in particular $\Gamma_i$ is torsion free, or $P_i$ is covered by $\dbS^3$ or by $\dbS^2\times \dbE$. Additionally, if $P_i$ is covered by $\dbS^2\times \dbE$, then $\Gamma_i$ is isomorphic to either $\dbZ$ or $D_\infty$ (see \cite[Table~1]{AFW15}).  Hence $V$ is isomorphic to either $\dbZ$ or $D_\infty$.

If $P_i$ has non-trivial JSJ-decomposition, thus by \Cref{lemma:JSJ:torsionfree}, $\Gamma_i$ is torsion free, hence $V$ must be isomorphic to $\dbZ$.
\end{proof}


\section{Computation of $H^{\Gamma}_*(\underline{E}\Gamma, E\Gamma; \dbK_R)$}\label{reltermf}
For a group $G$ consider the following properties.

\begin{itemize}
\item[(M)] Every non-trivial finite subgroup of $G$ is contained in a unique maximal
finite subgroup.
\item[(NM)] If $M$ is a maximal finite subgroup of $G$ then
$N_G(M)=M$, where $N_G(M)$ denotes the normalizer of $M$ in $G$.
\end{itemize}

A proof of the following two lemmas can be found in \cite{SS20}.

\begin{lemma}\label{lemma:M:NM}
Let $G$ be a group. Then the following two conditions are equivalent
\begin{enumerate}
    \item There exists a model $X$ for $\underline{E}G$ with the property that $X^H$ consists of exactly one point for every non-trivial finite subgroup $H$ of $G$. 
    \item Properties (M) and (NM) are true for $G$.
\end{enumerate}
\end{lemma}


\begin{lemma}\label{lemma:N:NM:3manifold}
Let $\Gamma=\pi_1(M)$ be a $3$-manifold group. 
Then $\Gamma$ satisfies properties (M) and (NM).
\end{lemma}

\begin{remark}
\Cref{lemma:N:NM:3manifold} together with \cite[Lemma~4.5]{BSS} provides an alternative proof for \Cref{prop:classify:vcyc}.
\end{remark}

The following theorem  generalizes \cite[Theorem~1.2]{Ro11}.


\begin{theorem}\label{finiterel}
Let $M$ be a closed, oriented, connected  3-manifold, and let $\Gamma$ be the fundamental group of $M$.  Let \[M=Q_1\#\cdots\#
Q_n\#P_1\#\cdots\#
P_m\] be the prime decomposition of $M$ such that  $Q_i$, $1\leq j \leq n$, are exactly those manifolds (if any) in the decomposition that are spherical. Denote $\Gamma_j=\pi_1(Q_j)$. Then, for all $i\in \dbZ$
\[
H^\Gamma_i(\underline{E}\Gamma, E\Gamma; \dbK_R)\cong \bigoplus_{j=1}^{n} Wh_i^R(\Gamma_j). 
\]
Equivalently, the splitting of $H^\Gamma_i(\underline{E}\Gamma, E\Gamma; \dbK_R)$ as a direct sum runs over the conjugacy classes of maximal finite subgroups of $\Gamma$.
In particular the \emph{classical} Whitehead group $Wh(\Gamma)=Wh_1(\Gamma)$ of $\Gamma$ is always a finitely generated abelian group.
\end{theorem}
\begin{proof}
Let $\calM$ be a set of representatives of conjugacy classes of finite maximal groups of $\Gamma$. By \Cref{lemma:N:NM:3manifold}, $\Gamma$ satisfies properties (M) and (NM). Then, we can use verbatim the proofs of \cite{DL03}, \cite{BSS}, or \cite{BJPP01} to show that, for all $i\in \dbZ$,
\[
H^\Gamma_*(\underline{E}\Gamma, E\Gamma; \dbK_R) \cong \bigoplus_{F\in \calM} Wh_i^R(F). 
\]
On the other hand, $\calM$ can be identified with the set of orbits of $0$-cells with nontrivial isotropy in the model for $\underline{E}\Gamma$ granted by the statement of proof of \Cref{lemma:N:NM:3manifold}. Thus, in view of \cref{prop:classify:torsion} every finite maximal subgroup of $\Gamma$ is either conjugated to a spherical factor of $\Gamma$ or subconjugated to a virtually cyclic (isomorphic to $\dbZ$ or $D_\infty$) of $\Gamma$. Since the Whitehead groups of all finite subgroups of $D_\infty$ are trivial or isomorphic to $\dbZ/2$, all their Whitehead groups vanish \cite{Wa78}.
Hence \[\bigoplus_{F\in\calM} Wh_i^R(F)\cong \bigoplus_{j=1}^{n} Wh_i^R(\Gamma_j).
\]

\end{proof}

Here is an amusing example.

\begin{example}
Let $M$ be the three dimensional Poi\-ncar\-\'e sphere. This is a spherical manifold and its fundamental group is the binary icosahedral group, $I^*$. By the calculations in \cite[Proposition 10, Theorem 12 (b) and Proposition 28] {GJM18} the lower algebraic $K$ groups of its group
ring are given as follows: $Wh(I^*)\cong \dbZ^2$, $\widetilde{K}_0(\dbZ[I^*])\cong \dbZ_2^3$, and  $K_{-1}(\dbZ[I^*])\cong \dbZ_2\oplus 
\dbZ^2$. Let $P$ be a prime aspherical 3-manifold, and let $\Gamma$ be the fundamental group of  M\#P, then the Whitehead groups of $G$ are given as follows
\[
Wh_i^{\dbZ}(\Gamma)\cong \begin{cases}
             \dbZ^2 &\text{for } i=1,\\
            \dbZ_2^3 &\text{for } i=0, \\
            \dbZ_2\oplus \dbZ^2 &\text{for } i=-1 \text{ and}\\
            0 &\text{for } i<-2.
            \end{cases}
\]            

J. Guaschi, D. Juan-Pineda and S. Mill\'an performed in \cite{GJM18}, extensive calculations  of lower $K$ theory groups of some of the groups that appear as fundamental groups of spherical 3-manifolds, one can manufacture examples with nontrivial
Whitehead groups using these calculations.

\end{example}


\section{Computations of $H^{\Gamma}_*(\underline{\underline{E}}\Gamma, \underline{E}\Gamma; \dbK_R)$: reducing to prime and JSJ-pieces} \label{RelTerm}

\subsection{The relative term $H^{\Gamma}_*(\underline{\underline{E}}\Gamma, \underline{E}\Gamma; \dbK_R)$ and the prime decomposition}

\begin{theorem}\label{reltermprimedecomposition}
Let $\Gamma=\pi_1(M)$ be a $3$-manifold group. Consider the prime decomposition $M=P_1\#\cdots \# P_m$ and the corresponding splitting $\Gamma\cong \Gamma_1 *\cdots * \Gamma_m$, where $\Gamma_i=\pi_1(P_i)$. Then the relative term $H^\Gamma_*(\underline{\underline{E}}\Gamma, \underline{E}\Gamma)$ fits in the long exact sequence 

\begin{align*}
 \cdots \to  \bigoplus_{i=1}^m H_n^{\Gamma_i}(\underline{\underline{E}} \Gamma_i, \underline{E}\Gamma_i)
 &\to H_n^\Gamma(\underline{\underline{E}}\Gamma, \underline{E}\Gamma) \to\\
 &\to \bigoplus_{\calH_1}2NK_n(R)\oplus \bigoplus_{\calH_2}NK_n(R) \to  \bigoplus_{i=1}^m H_{n-1}^{\Gamma_i}(\underline{\underline{E}} \Gamma_i, \underline{E}\Gamma_i)
   \to \cdots
 \end{align*}
where $\calH_1$ (resp. $\calH_2$) is a set of representatives of $\Gamma$-conjugacy classes of maximal elements in $\vcyc\setminus\calV$ that are isomorphic to $\dbZ$ (resp. $D_\infty$), and $\calV$ is the family of virtually cyclic subgroups of $\Gamma$ that are subconjugated to some $\Gamma_i$. 
\end{theorem}
\begin{proof}
Denote by $\mathbf{Y}$ any graph of groups associated to the splitting $\Gamma=\Gamma_1 *\cdots * \Gamma_m$, and let $T$ be the Bass-Serre tree of $\mathbf{Y}$. Hence we have a $\Gamma$-action on $T$ with trivial edge stabilizers and all vertex stabilizers isomorphic to some $\Gamma_i$.

Observe that $\calV$ consists of all virtually cyclic subgroups of $\Gamma$ that fix a point of $T$. Since every finite subgroup of $\Gamma$ fixes a vertex of $T$, we have the following chain of inclusions of families of $\Gamma$, $\fin\subseteq \calV\subseteq \vcyc $, that yield to the following inclusions
\[
\underline{E}\Gamma \to E_\calV \Gamma \to \underline{\underline{E}}\Gamma.
\]
 In other words we have the triple $(\underline{\underline{E}}\Gamma,E_\calV \Gamma, \underline{E}\Gamma)$. 
 The long exact sequence of the triple yields to the following long exact sequence
 \begin{align*}
 \cdots \to  H_n^\Gamma(E_\calV \Gamma, \underline{E}\Gamma) &\to H_n^\Gamma(\underline{\underline{E}}\Gamma, \underline{E}\Gamma) \to\\
 &\to H_n ^\Gamma(\underline{\underline{E}}\Gamma, E_\calV \Gamma)\to  H_{n-1}^\Gamma(E_\calV \Gamma, \underline{E}\Gamma)   \to \cdots .
 \end{align*}
 We now analyse the homology groups $H_n ^\Gamma(\underline{\underline{E}}\Gamma, E_\calV \Gamma)$ and  $H_n^\Gamma(E_\calV \Gamma, \underline{E}\Gamma)$.
 
 \vskip 10pt
 
 First we compute $H_n ^\Gamma(\underline{\underline{E}}\Gamma, E_\calV \Gamma)$. Since the edge stabilizers of the $\Gamma$-action on $T$ are trivial, the splitting of $\Gamma$ is \textit{acylindrical} in the sense of \cite[Definition~4.8]{JLSS}. Therefore one can construct a model for $\underline{\underline{E}}\Gamma$ using Proposition 4.9 from \cite{JLSS} as a (homotopy) $\Gamma$-pushout
\[
\begin{CD}
 \coprod_{H\in\mathcal{H}}\Gamma\times_{H}\mathbb{E}@>>>E_{\mathcal{V}}\Gamma\\
 @VVV@VVV\\
\coprod_{H\in\mathcal{H}}\Gamma\times_{H}\{ *\}@>>> \underline{\underline{E}}\Gamma .
 \end{CD}
\]
Since the above is a pushout, we have isomorphisms:
\begin{align*}
H^{\Gamma}_n( \underline{\underline{E}}\Gamma, E_{\mathcal{V}})& \cong H^{\Gamma}_n(\coprod_{H\in\mathcal{H}}\Gamma\times_{H}\mathbb{E}, \coprod_{H\in\mathcal{H}}\Gamma\times_{H}\{*\})\\
&\cong \bigoplus_{H\in\mathcal{H}} H^{\Gamma}_n(\Gamma\times_{H}\mathbb{E}, \Gamma\times_{H}\{*\}) \\
& \cong \bigoplus_{H\in\mathcal{H}} H^{H}_n(\mathbb{E}, \{*\}) \\
& \cong \bigoplus_{\calH_1}2NK_n(R)\oplus \bigoplus_{\calH_2}NK_n(R).  
\end{align*}

The last isomorphism  is a consequence of following two facts:
\begin{itemize}
    \item Every  $H\in \mathcal{H}$ is isomorphic to either $\dbZ$ or $D_\infty$. This was proved in \cref{prop:classify:vcyc}.
    
    \item We have isomorphisms \[
    H^{H}_n(\mathbb{E}, \{*\})\cong \begin{cases} 2NK_n(R) & \text{ if } H\cong \dbZ \\
    NK_n(R) & \text{ if } H\cong D_\infty ,
    \end{cases}
    \]
\end{itemize}
both isomorphisms can be found in \cite[Lemma~3.1]{DQR11}.

Now we deal with $H_n^\Gamma(E_\calV \Gamma, \underline{E}\Gamma)$. Let $\calG$ be the family of subgroups of $\Gamma$ generated by the isotropy groups of $T$. That is, $\calG$ is generated by the vertex groups of $\mathbf{Y}$. Then it is easy to see that the Bass-Serre tree $T$ is a one-dimensional model for $E_\calG \Gamma$. Hence using \cref{corollary:kles:relative:terms} and the fact that all edge isotropy groups of $T$ are trivial, and therefore $H_q^{\Gamma_e}(pt,\underline{E}\Gamma_e)=H_q^{\Gamma_e}(pt,pt)=0$ for all edge $e$ of $T$, and the fact that every vertex group satisfy the Farrell-Jones conjecture (see \cite{KLR16}) we get the isomorphisms
\begin{align*}
    H_n^\Gamma(E_\calG \Gamma, \underline{E}\Gamma) & \cong \bigoplus_{v\in V} H_n^{\Gamma_v}(pt, \underline{E}\Gamma_v)\\
    & \cong \bigoplus_{v\in V} H_n^{\Gamma_v}(\evc \Gamma_v, \underline{E}\Gamma_v)
\end{align*}
where $V$ is a set of representatives of the $\Gamma$-orbits of 0-cells in $T$.

On the other hand, we claim that  $H_n^\Gamma(E_\calG \Gamma, \underline{E}\Gamma)\cong H_n^\Gamma(E_\calV \Gamma, \underline{E}\Gamma)$. Consider the triple $(\fin, \calV, \calG)$ of families of $\Gamma$. Then we have the associated long exact sequence 
 \begin{align*}
 \cdots \to  H_n^\Gamma(E_\calV \Gamma, \underline{E}\Gamma) &\to H_n^\Gamma(E_\calG \Gamma, \underline{E}\Gamma) \to\\
 &\to H_n ^\Gamma(E_\calG\Gamma, E_\calV \Gamma)\to  H_{n-1}^\Gamma(E_\calV \Gamma, \underline{E}\Gamma)   \to \cdots .
 \end{align*}
We have to prove now that $H_n ^\Gamma(E_\calG\Gamma, E_\calV \Gamma)=0$. By \cref{cor:whitehead:ss} we have a spectral sequence that converges to $H_n ^\Gamma(E_\calG\Gamma, E_\calV \Gamma)$ and has as second page
\[
E^2_{p,q}= H_p(B_\calG \Gamma; \{ H_q^{\Gamma_{\sigma^p}}(pt,E_{\calV \cap \Gamma_{\sigma^p}}\Gamma_{\sigma^p})\}).
\]
Since $\Gamma_{\sigma^p}$ satisfies the Farrell-Jones conjecture and $\calV\cap \Gamma_{\sigma^p}$ is the family of virtually cyclic subgroups of $\Gamma_{\sigma^p}$ we conclude
\[H_q^{\Gamma_{\sigma^p}}(pt,E_{\calV \cap \Gamma_{\sigma^p}}\Gamma_{\sigma^p})\cong H_q^{\Gamma_{\sigma^p}}(\evc \Gamma_{\sigma^p},\evc \Gamma_{\sigma^p})=0. \]
Therefore $H_n ^\Gamma(E_\calG\Gamma, E_\calV \Gamma)=0$. This concludes the proof.
\end{proof}

Next, we want study the relative terms for prime manifolds.

\subsection{The relative term $H^{\Gamma}_*(\underline{\underline{E}}\Gamma, \underline{E}\Gamma; \dbK_R)$ and the JSJ-decomposition}

Our next task is to compute $H_n^{\Gamma}(\underline{\underline{E}} \Gamma, \underline{E}\Gamma)$ where $\Gamma$ is the fundamental group of a prime $3$-manifold $N$. Before stating the main result of this subsection, we need to stablish some notation.

Denote by $\mathbf{X}$ the graph of groups associated to the splitting of $\Gamma$ given by the JSJ-decomposition of $N$ (see \cref{prime:geometric:splittings}), and let $S$ be the Bass-Serre tree. Therefore we have a $\Gamma$-action on $S$ with edge stabilizers isomorphic to $\dbZ^2$ and all vertex stabilizers are isomorphic to the fundamental group of a Seifert fibered $3$-manifold or to the fundamental group of a hyperbolic $3$-manifold (see \cref{jsj decomposition}).

Define the family $\calW$ consisting of all virtually cyclic subgroups of $G$ that are subconjugated to a vertex group of $\mathbf{X}$. Denote the connected components that appear in the statement of \cref{jsj decomposition} as $N_1,\dots, N_m$, and their fundamental groups $\Gamma_1,\ldots, \Gamma_m$. These groups groups are precisely the vertex groups of $\mathbf{X}$.

\begin{theorem}\label{reltermjsjdecomposition}
We keep using the notation as above. Let $N$ be a prime manifold with fundamental group $\Gamma$ and nontrivial JSJ-decomposition. Assume that the minimal JSJ-decomposition of $N$ is not a double of a twisted $I$-bundle over the Klein bottle or a mapping torus of an Anosov homeomorphism of the 2-dimensional torus.
Then the relative term $H^{\Gamma}_*(\underline{\underline{E}}\Gamma, \underline{E}\Gamma)$ fits in the long exact sequence 
 
 \begin{align*}
 \cdots \to  H_n^{\Gamma}(E_\calW \Gamma, \underline{E}\Gamma)
 &\to H_n^\Gamma(\underline{\underline{E}}\Gamma, \underline{E}\Gamma) \to\\
 &\to \bigoplus_{\calH_1}2NK_n(R)\oplus \bigoplus_{\calH_2}NK_n(R) \to  H_{n-1}^{\Gamma}(E_\calW \Gamma, \underline{E}\Gamma)
   \to \cdots
 \end{align*}
where $\calH_1$ (resp. $\calH_2$) is a set of representatives of $\Gamma$-conjugacy classes of maximal elements in $\vcyc\setminus\calW$ that are isomorphic to $\dbZ$ (resp. $D_\infty$), and $\calW$ is the family of virtually cyclic subgroups of $\Gamma$ that are subconjugated to some vertex group of $\mathbf{X}$. 
Moreover, the term  $H_n^{\Gamma}(E_\calW \Gamma, \underline{E}\Gamma)$ fits in the long exact sequence
\begin{align*}
\cdots \to \bigoplus_{E(S)}\bigoplus_{1}^\infty (2NK_n(R) \oplus 2NK_{n-1}(R))
 &\to \bigoplus_{ i=1}^m H_n^{\Gamma_i}(\underline{\underline{E}}\Gamma_i, \underline{E}\Gamma_i)\\
&\to H_n^G(E_\calW \Gamma,\underline{E}\Gamma)\to \cdots \end{align*}
where $E(S)$ is the edge set of $\mathbf{X}$.
\end{theorem}
\begin{proof}
The proof of this theorem is very similar to the proof of \cref{reltermprimedecomposition}. We only highlight the key points and leave the details to the reader. 

From the triple $(\efin \Gamma, E_\calW \Gamma,\evc \Gamma)$ we get the  long exact sequence 
 \begin{align*}
 \cdots \to  H_n^{\Gamma}(E_\calW \Gamma, \underline{E}\Gamma) &\to H_n^{\Gamma}(\underline{\underline{E}}\Gamma, \underline{E}\Gamma) \to\\
 &\to H_n ^{\Gamma}(\underline{\underline{E}}\Gamma, E_\calW \Gamma)\to  H_{n-1}^{\Gamma}(E_\calW \Gamma, \underline{E}\Gamma)   \to \cdots 
 \end{align*}
 
 We can prove that  $H_n^{\Gamma} (\underline{\underline{E}}\Gamma, E_\calW \Gamma)\cong \bigoplus_{\calH_1}2NK_n(R)\oplus \bigoplus_{\calH_2}NK_n(R)$  using the acylindricity of the splitting $\Gamma=\pi_1(\mathbf{X})$ which is proved in \cite[Proposition~8.2]{JLSS}. Note that here is where we are using the hypothesis that the JSJ-decomposition of $N$ is not the double of a twisted I-bundle over the Klein bottle. 

We proceed as in the proof of \cref{reltermprimedecomposition} to show that 
$$
H_n^{\Gamma}(E_\calW \Gamma, \underline{E}\Gamma)\cong H_n^{\Gamma}(E_\calG \Gamma,  \underline{E}\Gamma)
$$
where $\calG$ is the family of subgroups of  $\Gamma$ generated by the vertex groups of $\mathbf{X}$. 

For the \textit{moreover} part of the statement we proceed as follows. By \cref{corollary:kles:relative:terms}, and the fact that every edge group of $\mathbf{X}$ is isomorphic to $\dbZ^2$, the group $H_n^{\Gamma}(E_\calG \Gamma,  \underline{E}\Gamma)$ fits in the long exact sequence 
\begin{align*}
\cdots \to \bigoplus_{E(S)} H_n^{\dbZ^2}(\evc \dbZ^2, E\dbZ^2)
 \to \bigoplus_{ i=1}^m H_n^{\Gamma_i}(\underline{\underline{E}}\Gamma_i, \underline{E}\Gamma_i)
\to H_n^{\Gamma}(E_\calW \Gamma,\underline{E}\Gamma)\to \cdots .\end{align*} 
As a final step we have the isomorphism \[H_n^{\dbZ^2}(\evc \dbZ^2, E\dbZ^2)\cong \bigoplus_{1}^\infty (2NK_n(R) \oplus 2NK_{n-1}(R))\] from \cite[Theorem~1]{Da08} (see also  \cref{prop:relative:term:Z2}). This finishes the proof.
\end{proof}

\subsection{The exceptional cases}\label{section:exceptional:cases}

In \cref{reltermjsjdecomposition} there are two exceptional cases: 
\begin{itemize}
    \item $N$ is the double of a twisted $I$-bundle over the Klein bottle. Such a bundle is denoted $\calK \tilde \times I$ and it has exactly one boundary component isomorphic to the 2-dimentional torus $T^2$.
    \item $N$ is the mapping torus of an Anosov homeomorphism of the 2-dimensional torus.
\end{itemize}

In the first exceptional case three things can happen, deppending on the homeomorphism used to glue the boundary components of the two coppies of $\calK \tilde \times I$. The resulting manifold $N$ is modeled on $\mathbb{E}^3$, $\mathrm{Nil}$ of $\mathrm{Sol}$. On the other hand, any manifold modeled on  $\mathbb{E}^3$ or $\mathrm{Nil}$ are already Seifert, hence this cases are not exeptional cases in \cref{reltermjsjdecomposition} as the (minimal) JSJ-decomposition of $N$ tuns out to be trivial. Hence in this case $N$ must be modeled on $\mathrm{Sol}$ (see \cite[Lemma~1.5.5]{AFW15}). In the second exceptional case $N$ is also modeled on $\mathrm{Sol}$.
Threrefore the description of $H^{\Gamma}_n(\evc{\Gamma},E\Gamma)=Wh_n^R(\Gamma)$ can be seen as a particular case of \cref{relative:term:virtually:solvable:groups}.


\section{The relative term $H^{\Gamma}_*(\underline{\underline{E}}\Gamma, \underline{E}\Gamma; \dbK_R)$ for Seifert $3$-manifolds}\label{reltermseifert}
Finally, we have to deal with the homology groups $H_n^{\Gamma}(\underline{\underline{E}}\Gamma, \underline{E}\Gamma)$, where $\Gamma=\pi_1(W)$, with $W$ either a Seifert fibered $3$-manifold or a hyperbolic $3$-manifold.


\subsection{Seifert fibered case with spherical orbifold base}\label{section:Sefiert:spherical}
Let $M$ be a closed Sei\-fert fiber space with base orbifold $B$ and fundamental
group $\Gamma$. Assume that $B$ is either a bad orbifold, or a good orbifold modeled on $S^2$. Then, by \cite[Theorem~1.2.2]{Mo05}, $M$ is modeled on $\dbS^3$ or $\dbS^2\times \dbE$. In the former case $\Gamma$ is finite (see \cref{prop:classify:torsion}), and in the latter case $\Gamma$ is isomorphic wither to $\dbZ$ or $D_\infty$ by \cite[Table~1]{AFW15}. Hence, by the results in \cref{subsubsectio:relative:terms} we get

\[
H^{\Gamma}_n(\evc{\Gamma},\efin{\Gamma})\cong \begin{cases}
0 & \text{ if } \Gamma \text{ is finite}\\
2NK_n(R) & \text{ if } \Gamma\cong \dbZ \\
    NK_n(R) & \text{ if } \Gamma\cong D_\infty
    \end{cases}
\]

For the rest of this section, all groups are torsion-free since they are the fundamental groups of aspherical 3-manifolds (see \cite[Theorem~1.2.2]{Mo05} for instance). Hence the relative term $H^{\Gamma}_n(\evc{\Gamma},\efin{\Gamma})=H^{\Gamma}_n(\evc{\Gamma},E\Gamma)$ is by definition isomorphic to $Wh_n^R(\Gamma)$ for all $n\in \dbZ$.


\subsection{Seifert fibered case with flat orbifold base and nonempty boundary}\label{section:seifert:flat}

\begin{proposition}{\cite[Proposition 5.6]{JLSS}}\label{prop:bdry:flat:orb}
Let $M$ be a compact Seifert fibered manifold with nonempty boundary. Let $\Gamma=\pi_1(M)$, and let $B$ be the base orbifold of $M$. If $B^\circ$ is modeled on $\dbE^2$, then $\Gamma$ is $2$-crystallographic isomorphic to $\dbZ^2$ or $\dbZ\rtimes \dbZ$.
\end{proposition}

The following result is a particular case of \cite[Theorem~1]{Da08} and identifies the relative term for the group $\dbZ\times \dbZ$.  

\begin{proposition}\label{prop:relative:term:Z2}
The relative term $H^{\dbZ\times \dbZ}_n(\evc{\dbZ^2},\efin{\dbZ^2})$ is by definition  $Wh_i^R(\dbZ^2)$ and we have isomorphisms 
\[
Wh_i^R(\dbZ^2)\cong \bigoplus_{j=0}^{\infty}(2NK_i(R)\oplus 2NK_{i-1}(R)).
\]
\end{proposition}

The computation of the relative term in the case of the Klein bottle group is given by the following proposition.

\begin{proposition}\label{prop:relative:term:K}
Let $\calK=\kleingp$ the fundamental group of the Klein bottle. Then we have the following isomorphisms

\[
Wh^R_i(\calK)\cong H^\calK_i(\underline{\underline{E}}\calK, \underline{E}\calK) \cong  \bigoplus_{j=0}^{\infty}( NK_i(R)\oplus NK_{i-1}(R) ).
\]

\end{proposition}
\begin{proof}
In order to compute $H^\calK_*(\underline{\underline{E}}\calK, \underline{E}\calK)$, we will use the model for $\underline{\underline{E}}\calK$ described in \cite[Section~4.2]{JPTN18}. Such a model is given by the following $\calK$-(homotopy)-pushout:

\[\xymatrix{ E\calK \coprod E\calK \coprod_{R\in\calI}\calK\times_{\calK}E\dbZ^2 
\ar[d] \ar[r]  & \underline{E}\calK=E \calK \ar[d]\\
E_{SUB(H)}\calK \coprod E_{\calG[K]}\calK \coprod_{R\in\calI}\calK\times_{\calK}E_{SUB(R)}\dbZ^2 \ar[r] & \evc \calK
}
\]

where
\begin{enumerate}
    \item $H$ is the subgroup of $\calK$ generated by $(1,0)\in \calK$.
    \item $K$ is the subgroup of $\calK$ generated by $(0,2)\in \calK$.
    \item $\calI$ is a set of representatives of conjugacy classes of maximal cyclic subgroups of $\calK$ generated by elements of the form $(n,2m)$. In particular $\calI$ is an infinite numerable set.
    \item $SUB(H)$ is the family of all subgroups of $H$. 
    \item $\calG[K]$ is the family consisting of all cyclic subgroups $L$ of $\calK$ such that either $L$ is trivial of $L\cap K$ is infinite.
    \item The left vertical arrow is the disjoint union of the maps $E\calK\to E_{SUB(H)}\calK $, $E\calK \to E_{\calG[K]} \calK$, and $\calK\times_{\calK}E\dbZ^2 \to \calK\times_{\calK}E_{SUB(R)}\dbZ^2$.
\end{enumerate}

As an immediate consequence $H_i^{\calK}(\underline{\underline{E}}\calK, \underline{E}\calK)$ is isomorphic to 

\[
H_i^\calK(E_{SUB(H)}\calK,E\calK)\oplus  H^\calK_i(E_{\calG[K]} \calK, E\calK) \oplus \bigoplus_{R\in\calI}H_i^{\dbZ^2}(E_{SUB(H)}\dbZ^2,E\dbZ^2).
\]
 To finish the proof we will describe each of these factors. 
 Let $X=\{K_n\leq \calK\colon K_n=\langle (n,1)\rangle,n\in \dbZ \}$ endowed with the discrete topology. Note that $\calK$ acts on $X$ by conjugation and the isotropy group of $K_n\in X$ is the normalizer $N_\calK (K_n)$ which is isomorphic to $\dbZ$ (see \cite[p.~354]{JPTN18}).
By \cite[p.~354]{JPTN18}, a model for $E_{\calG[K]}\calG$ is given by the join $X*E\calK$ endowed with the diagonal $\calK$-action. Therefore, a computation completely analogous to that in \cite[Section~5]{JPL06} yields the following isomorphism
\[
H_i^\calK(E_{\calG[K]}\calK, E\calK)\cong \bigoplus_{X/G} 2NK_i(R).
\]
Since the circle $S^1$ is a model for $E_{SUB(H)}\calK$ and $E_{SUB(H)}\dbZ^2$, an application of the spectral sequence described in \Cref{cor:whitehead:ss} leads to the following isomorphisms

\begin{align*}
    H_i^\calK(E_{SUB(H)}\calK,E\calK) &\cong Wh^R_i(\dbZ)\oplus Wh_{i-1}^R(\dbZ)\\
    &\cong 2NK_i(R)\oplus 2NK_{i-1}(R),
\end{align*}
and analogously
\[
H_i^{\dbZ^2}(E_{SUB(H)}\dbZ^2,E\dbZ^2) \cong 2NK_i(R) \oplus 2NK_{i-1}(R).
\]

\end{proof}


\subsection{The Whitehead groups of $\dbZ^2\rtimes \dbZ$ and $\calK \ast_{\dbZ^2}\calK$}

Before we describe the Whitehead groups of Seifert fibered 3-manifolds with flat base orbifold, we describe the Whitehead groups of the form $\dbZ^2\rtimes \dbZ$ and $\calK \ast_{\dbZ^2}\calK$. The reason is that the fundamental groups of closed Seifert fibered 3-manifolds with flat base orbifold essentially have this algebraic description. Actually, the fundamental groups of manifolds modeled on $\mathrm{Sol}$ also have these forms.

The following proposition is a straightforward consequence of the results in \cref{subsubsectio:relative:terms}, \cref{prop:relative:term:Z2}, and \cref{prop:relative:term:K}.

\begin{proposition}\label{relative:term:virtually:solvable:groups}
Let $G$ be a group. 
\begin{enumerate}
    \item  If $G$ isomorphic to the semidirect product $\dbZ^2\rtimes_\phi \dbZ$ with $\phi$ an automorphism of $\dbZ^2$, then for all $n\in \dbZ$ we get
\[
Wh^R_n(G)\cong \bigoplus_{j=0}^{\infty}\Big(2NK_i(R)\oplus 4NK_{i-1}(R) \oplus 2NK_{i-2}(R)\Big)\oplus 2NK_n(\dbZ^2; \phi).
\]

    \item If $G$ isomorphic to the amalgamated product $\calK\ast_{\dbZ^2} \calK$ with $\dbZ^2$ embedded in each copy of $\calK$ as a subgroup of index 2. Let let $\calF$ be the smallest family of $G$ containing the two copies of $\calK$. Then for all $n\in \dbZ$ we get
\[
Wh^R_n(G)\cong H_n^G(E_\calF G, EG) \oplus NK_n(\dbZ^2; \phi).
\]
and the term $H_n^G(E_\calF G, EG)$ fits in the following long exact sequence
\begin{align*}
\cdots \to\bigoplus_{j=0}^{\infty}(2NK_i(R)\oplus & 2NK_{i-1}(R))\to  2\bigoplus_{j=0}^{\infty}( NK_i(R)\oplus NK_{i-1}(R) )\to \\
\to & H_n^G(E_\calF G, EG)\to \cdots .
\end{align*}
\end{enumerate}
\end{proposition}


\subsection{Seifert fibered case with flat orbifold base and empty boundary}
Let $M$ be a Seifert fibered manifold with flat base orbifold, and let $\Gamma$ be its fundamental group. By \cite[Theorem~1.2.2]{Mo05} $M$ is modeled on either $\dbE^3$ or $\mathrm{Nil}$. Thus the description of $H^{\Gamma}_n(\evc{\Gamma},E\Gamma)=Wh_n^R(\Gamma)$ can be seen as a particular case of \cref{relative:term:virtually:solvable:groups}.


\subsection{Seifert fibered case with hyperbolic orbifold base}\label{section:seifert:hyperbolic}
\newcommand{\mc}[1]{\mathcal{#1}}
\newcommand{\hyp}{\mathbb{H}^2}

Before we provide a description for $H_n^{\Gamma}(\underline{\underline{E}} \Gamma, \underline{E}\Gamma)=H_n^{\Gamma}(\underline{\underline{E}} \Gamma, E\Gamma)$, we need to state some notation and preliminary results.

Let $M$ be a Seifert fibered space with base orbifold $B$ modeled on
$\hyp$.  Let $\Gamma=\pi_1(M)$ and $\Gamma_0=\pi_1(B)$ be the respective
fundamental groups, and let $K$ be the infinite cyclic subgroup of $\Gamma$ generated by a regular fiber of $M$. Thus we have the short exact sequence

\[1\to K \to \Gamma \to \Gamma_0 \to 1.\]

Let $\mc{A}$ be the collection of maximal infinite virtually cyclic
subgroups of $\Gamma_0$, let $\widetilde{\mc{A}}$ be the collection of preimages
of $\mc{A}$ in $\Gamma$, and let $\mc{H}$ be a set of representatives of
conjugacy classes in $\widetilde{\mc{A}}$. For an element $H\in \mathcal A$ we denote by $\tilde H$ the corresponding element in $\widetilde{\mc{A}}$. Let $\calF'$ be the smallest family that contains all infinite cyclic subgroups of $\Gamma$ that map to a finite group in $\Gamma_0$.

Consider the following homotopy
cellular homotopy $\Gamma$-push-out:

\begin{equation}\label{pushout:Seifert:hyperbolic}
\xymatrix{\coprod_{\tilde{H}\in\calH}\Gamma\times_{\tilde{H}}E_{\tilde H\cap\calF'}\tilde H \ar[d] \ar[r]  & E_{\calF'}\Gamma \ar[d]\\
\coprod_{\tilde{H}\in\calH}\Gamma\times_{\tilde{H}}\evc\tilde{H}\ar[r] & X
}
\end{equation}
Then by \cite[Proposition~5.12~and~Proposition~5.7]{JLSS} $X$ is a model for $\evc\Gamma$. 

Next, we need extra information about the terms involved in the above $\Gamma$-pushout.

\begin{lemma}\label{lemma:pushout:Seifert:hyp}
With the notation as above. The following statements are true:

\begin{enumerate}
    \item The hyperbolic plane is a model for $E_{\calF'}\Gamma$, and the quotient space $\dbH^2/\Gamma=\dbH^2/\Gamma_0$ is the base orbifold $B$.
    
    \item Each $\tilde H\in \tilde\calA$ is isomorphic to either $\dbZ^2$ or $\calK$.
    
    \item The family $\tilde H\cap \calF'$ is the family $SUB(K)$ consisting of the subgroups of $K$.
\end{enumerate}
\end{lemma}
\begin{proof}
Let us prove the first statement. The family $\calF'$ is the pull-back family in $\Gamma$ of the family of finite subgroups of $\Gamma_0$. Since $\underline{E}\Gamma_0=\dbH^2$, we get that the hyperbolic plane is a model for $E_{\calF'}\Gamma$, and the quotient space $\dbH^2/\Gamma=\dbH^2/\Gamma_0$ is the base orbifold $B$.

Let us now prove the second statement. Since $\Gamma$ is torsion-free, every $\tilde H$ is also torsion free. Moreover, each $H\in \mathcal A$ acts as a hyperbolic isometry on $\dbH^2$ thus there is unique geodesic $\gamma$ upon which $H$ is acting. Therefore $\tilde H$ acts by isometries on the preimage of $\gamma$ in $\tilde M$ which turns out to be a flat. Therefore     each $\tilde H$ is a torsion-free 2-crystallographic group. Now the claim follows by the theorem of classification of closed surfaces.

For the third statement, let us note that, by part 2 of the lemma, $\tilde H$ fits in the long exact sequence
\[1\to K\to \tilde H\to \dbZ \to 1.\]
Now it is easy to see that the only subgroups of $\Gamma$ that map to a finite group of $\Gamma_0$ are exactly the subgroups of $K$.
\end{proof}

\begin{theorem}\label{thm:rel:term:hyperbolic}
With the notation as above. The relative term $H_n^{\Gamma}(\underline{\underline{E}} \Gamma, \underline{E}\Gamma)=H_n^{\Gamma}(\underline{\underline{E}} \Gamma, E\Gamma)$ fits in the long exact sequence
\[\cdots \to H_n^\Gamma(E_{\calF'}\Gamma,E\Gamma) \to H_n^\Gamma(\underline{\underline{E}}\Gamma,E\Gamma)\to \bigoplus_{\tilde H\in \tilde \calA}H_n^{\tilde H}(\underline{\underline{E}}\tilde H, E_{SUB(K)}\tilde H) \to \cdots \]
where each $\tilde H$ is either isomorphic to $\dbZ^2$ or to $\calK$.
\end{theorem}

\begin{proof}
The statement is a direct consequence of \eqref{pushout:Seifert:hyperbolic}, \cref{lemma:pushout:Seifert:hyp} and the long exact sequence of the triple $(\underline{\underline{E}}\Gamma, E_{\calF'}\Gamma, E\Gamma)$. 

\end{proof}

The previous theorem provides a description of $H_n^\Gamma(\underline{\underline{E}}\Gamma,E\Gamma)$ in terms of $H_n^\Gamma(E_{\calF'}\Gamma,E\Gamma)$ and the terms $H_n^{\tilde H}(\underline{\underline{E}}\tilde H, E_{SUB(K)}\tilde H)$. To finish this section we explain possible approaches to compute these terms.

The term $H_n^\Gamma(E_{\calF'}\Gamma,E\Gamma)$ can be computed using the spectral sequence of \cref{cor:whitehead:ss}, that is, we gave a spectral sequence that converges to $H_n^\Gamma(E_{\calF'}\Gamma,E\Gamma)$ and whose second page is given by
\[
E^2_{p,q}=H_p(\dbH^2/\Gamma_0; \{ \{ Wh_q^R((\Gamma_0)_\sigma) \} \})=H_p(B; \{ Wh_q^R((\Gamma_0)_\sigma) \}).
\]
Moreover, each $(\Gamma_0)_\sigma$ is isomorphic to $\dbZ$, therefore $Wh_q^R((\Gamma_0)_\sigma)\cong 2NK_q(R)$. 

The term $H_n^{\tilde H}(\underline{\underline{E}}\tilde H, E_{SUB(K)}\tilde H)$ fits in the long exact sequence
\begin{equation}\label{eq:H:tilde:sequence}\cdots \to H_n^{\tilde H}( E_{SUB(K)}\tilde H) \to H_n^{\tilde H}(\underline{\underline{E}}\tilde H)  \to H_n^{\tilde H}(\underline{\underline{E}}\tilde H, E_{SUB(K)}\tilde H) \to  \cdots  .
\end{equation}
Since $E_{SUB(K)}\tilde H)=E(H/K)=E\dbZ$, we get \begin{align}
H_n^{\tilde H}( E_{SUB(K)}\tilde H)&\cong K_n(R[K])\oplus K_{n-1}(R[K])\nonumber\\
&\cong K_n(R)\oplus 2 K_{n-1}(R)\oplus K_{n-2}(R)\oplus 2NK_n(R) \oplus 2NK_{n-1}(R)\label{eq:H:tilde:first:term}.
\end{align}
In case $\tilde H\cong \dbZ^2$, by \cite[Corollary~2]{Da08} we have the following isomorphism
\begin{equation}\label{eq:H:tilde:second:term}
K_n(R[\dbZ^2])\cong K_n(R)\oplus 2 K_{n-1}(R)\oplus K_{n-2}(R)\oplus \bigoplus_{i=0}^{\infty} (2NK_n(R)\oplus 2NK_{n-1}(R)).
\end{equation}

From \eqref{eq:H:tilde:sequence}, \eqref{eq:H:tilde:first:term}, and \eqref{eq:H:tilde:second:term} we get \[H_n^{\dbZ^2}(\underline{\underline{E}}\dbZ^2, E_{SUB(\dbZ\oplus 1)}\dbZ^2)\cong \bigoplus_{i=1}^{\infty} (2NK_n(R)\oplus 2NK_{n-1}(R)).\]

Sadly we do not have a description for $H_n^{\calK}(\underline{\underline{E}}\calK, E_{SUB(\dbZ\oplus 1)}\calK)$ as good as the above one, although we suspect this term should be very similar in nature as $H_n^{\dbZ^2}(\underline{\underline{E}}\dbZ^2, E_{SUB(\dbZ\oplus 1)}\dbZ^2)$.

\section{The relative term $H^{\Gamma}_*(\underline{\underline{E}}\Gamma, \underline{E}\Gamma; \dbK_R)$ for hyperbolic 3-manifolds} \label{reltermhyp}

Let $\Gamma$ be the fundamental group of a hyperbolic manifold $M$ with possibly non-empty boundary. Since $M$ is aspherical, we conclude that $\Gamma$ is torsion-free.

Let $\calA$ be the collection of infinite maximal subgroups $M_c$ that stabilize a
geodesic $c(\dbR)\subset\dbH$ and infinite maximal parabolic subgroups
$P_{\xi}$ that fix a unique boundary point $\xi\in\partial\dbH$. Let $\calH$ be a set of representative of conjugacy classes of maximal infinite virtually cyclic (actually cyclic since $\Gamma$ is torsion-free) subgroups of $\Gamma$. We identify the relative term as follows:
\begin{theorem}
Let $\Gamma$ be the fundamental group of a hyperbolic manifold with possibly non-empty boundary. Then, 
\begin{align*}
H^{\Gamma}_n(\evc{\Gamma},\efin{\Gamma})
\cong \bigoplus_{H\in\mathcal{H}\cap \mathcal{M}_c}2NK_n(R) \oplus \bigoplus_{H\in\mathcal{H}\cap P_{\xi}}
(2NK_n(R)\oplus 2NK_{n-1}(R)).
\end{align*}
\end{theorem}

\begin{proof}
Let $I$ be a
complete set of representatives of the conjugacy classes within $\calA$, and
consider the following cellular homotopy $\Gamma$-push-out:
\[
\xymatrix{\coprod_{H\in I}\Gamma\times_{H}\underline{E}H\ar[d] \ar[r]  & \underline{E}\Gamma=\dbH^3 \ar[d]\\
\coprod_{H\in I}\Gamma\times_{H}\evc H\ar[r] & X
}
\]
Then $X$ is a $3$-dimensional model for $\underline{\underline{E}}\Gamma$, see Proposition 6.1 in \cite{JLSS}.
The above gives the isomorphism
\begin{align*}
H^{\Gamma}_n(\evc{\Gamma},\efin{\Gamma})&\cong \bigoplus_I H_n^H(\underline{\underline{E}}H, \underline{E}H)\\
&\cong \bigoplus_{H\in\mathcal{H}\cap \mathcal{M}_c}2NK_n(R) \oplus \bigoplus_{H\in\mathcal{H}\cap P_{\xi}}
(2NK_n(R)\oplus 2NK_{n-1}(R)).
\end{align*}
and the last isomorphism follows from the fact that each $H$ is isomorphic either to $\dbZ$ or $\dbZ^2$, \cref{prop:relative:term:Z2} and the isomorphism $H_n^{\dbZ}(\underline{\underline{E}}\dbZ, \underline{E}\dbZ)\cong 2NK_n(R)$.

\end{proof}


\bibliographystyle{alpha} 
\bibliography{myblib}
\end{document}